\documentclass[a4paper,11pt]{article} 

\usepackage[english]{babel}

\usepackage{anysize}
\marginsize{2.8cm}{2.8cm}{2.5cm}{2.5cm}
\usepackage{fancyhdr} 
\usepackage{amsfonts}
\usepackage{amsmath}
\usepackage{mathtools}
\usepackage{mathrsfs} 
\usepackage{tikz,graphicx,subfigure,overpic}
\usepackage{color} 
\usepackage{amssymb} 
\usepackage{amsthm}
\usepackage[colorlinks=true, linkcolor=blue, citecolor=blue, urlcolor=blue]{hyperref}
\usepackage[fixlanguage]{babelbib}
\usepackage{color}
\usepackage{comment}
\usepackage{subfigure}
\usepackage{authblk}
\numberwithin{equation}{section}

\newtheorem{theorem}{Theorem}[section]
\newtheorem{lemma}[theorem]{Lemma}
\newtheorem{corollary}[theorem]{Corollary}
\newtheorem{proposition}[theorem]{Proposition}

\theoremstyle{definition} 

\newtheorem{Remark}[theorem]{Remark}
\newtheorem{OP}[theorem]{Open problem}
\newenvironment{remark}{\begin{Remark}\rm}{\end{Remark}}
\newenvironment{example}{\begin{Example}\rm}{\end{Example}}
\newtheorem{Example}[theorem]{Example}

\newcommand{\eq}{\begin{equation}}
\newcommand{\qe}{\end{equation}}
\newcommand{\beqa}{\begin{eqnarray}}
\newcommand{\eeqa}{\end{eqnarray}}
\newcommand{\beqan}{\begin{eqnarray*}}
\newcommand{\eeqan}{\end{eqnarray*}}

\newcommand{\N}{\mathbb{N}}
\newcommand{\Z}{\mathbb{Z}}
\newcommand{\R}{\mathbb{R}}
\newcommand{\C}{\mathbb{C}}
\newcommand{\E}{\mathbb{E}}
\renewcommand{\P}{\mathbb{P}}

\newcommand{\cO}{\mathcal{O}}

\newcommand{\Tr}{{\rm Tr}}

\newcommand{\Supp}{{\rm Supp}}
\newcommand{\Span}{{\rm Span}}

\newcommand{\bs}{\boldsymbol}
\newcommand{\bv}{\mathbf}

\newcommand{\Var}{\mathbb{V}{\mathrm{ar}}}

\renewcommand{\phi}{\varphi}
\renewcommand{\epsilon}{\varepsilon}
\renewcommand{\d}{ {\rm d}}
\renewcommand{\emptyset}{\varnothing}
\renewcommand{\Re}{\mathfrak{Re}}

\usepackage{natbib}

\usepackage{url}
\DeclareUrlCommand\email{\urlstyle{rm}}

\usepackage{url}
\DeclareUrlCommand\email{\urlstyle{rm}}

\title{Polynomial Ensembles and Recurrence Coefficients} 

\author{\ Adrien Hardy \footnote{Laboratoire Paul Painlev\'e, Universit\'e de Lille,
  Cit\'e Scientifique, 59655 Villeneuve d'Ascq Cedex, France. Email: \href{mailto:adrien.hardy@math.univ-lille1.fr}{\nolinkurl{adrien.hardy@math.univ-lille1.fr}}}
}

\begin{document}
\maketitle 

\begin{abstract}
Polynomial ensembles are determinantal point processes associated with (non necessarily orthogonal) projections onto polynomial subspaces. The aim of this survey article is to put forward the use of recurrence coefficients to obtain the global asymptotic behavior of such ensembles in a rather simple way.  We provide a unified approach to recover well-known convergence results for real OP ensembles. We study the mutual convergence of the polynomial ensemble and the zeros of its average characteristic polynomial; we  discuss in particular the complex setting.  We also control the variance of linear statistics of polynomial ensembles and derive comparison results, as well as  asymptotic formulas for real OP ensembles. Finally, we reinterpret the classical algorithm to sample determinantal point processes so as to cover the setting of non-orthogonal projection kernels.  A few open problems are also suggested. 
\end{abstract}

\section{Introduction}

Determinantal point processes (DPPs) are random point configurations where the points tend to repel each others.  A DPP is parametrized by a kernel $K(x,y)$ and a reference measure $\mu$ on the space $\Lambda$ where the points, or particles, live. The most repulsive DPPs arise when the integral operator $K$ acting on $L^2(\mu)$ associated with the kernel is a projection operator, in which case the number of points in the configuration equals to the rank of the projection; we refer to  \citep{Sos00b, Lyo03, Joh06, HoKrPeVi06} for general presentations. We here focus on polynomial ensembles, which are DPPs  coming  with a projection $K_N$ onto polynomials of degree less than $N$. Such models are special instances of biorthogonal ensembles \citep{Bor99}. When $K_N$ is an orthogonal projection, this yields an important class of DPPs popularized by random matrix theory among other things: the orthogonal polynomial (OP) ensembles \citep{Kon05}. Non-orthogonal projections also appear in several interesting models. For instance multiple OP ensembles, which involve polynomials satisfying multiple orthogonality conditions with respect to several scalar products, appear in several non-unitary invariant random matrix models \citep{Kui10}. Polynomial ensembles turn out to be a class of models which is stable under natural matrix operations \citep{Kui16}. 

 \subsection{Polynomial ensembles} 
Let $\mu$ be a Borel measure with infinite support $\Lambda\subset\C$ such that any polynomial belongs to $L^2(\mu)$. Consider two families $P_{k}$ and $Q_{k}$ of function in $L^{2}(\mu)$ such that $P_{k}$ is a polynomial of degree $k$ and the $Q_{k}$'s satisfy the biorthogonality relations
\eq
\label{biortho}
\langle P_k,Q_m\rangle :=\int P_{k} (x)\overline{Q_{m}}(x) \mu(\d x)=\delta_{k,m},\qquad k,m\in \N.
\qe
A \emph{polynomial ensemble} (of $N$ points) is a probability distribution on $\Lambda^N$ of the form
\eq
\label{DPPdistr}
\d \P(x_1,\ldots,x_N)=\frac{1}{N!}\det\Big[K(x_i,x_j)\Big]_{i,j=1}^N\prod_{i=1}^N\mu(\d x_i),
\qe
with kernel 
\eq
\label{kernel}
K(x,y)=\sum_{k=0}^{N-1}P_{k}(x)\overline{Q_{k}}(y).
\qe
When the reference measure is supported on $\R$, we say that $\P$ is a \emph{real polynomial ensemble}.

It is known that a polynomial ensemble $\P$ induces a DPP with kernel $K(x,y)$, namely for every $k\geq 1$ and every positive Borel function $f:\Lambda^{k}\to\C$, we have
\eq
\label{def k cor}
\mathbb E\left[\sum_{{i_1}\neq \cdots\neq {i_k}}f(  x_{i_1}, \ldots, x_{i_k})\right]=\int_{\Lambda^k}f(x_1,\ldots,x_k)\det\Big[K(x_i,x_j)\Big]_{i,j=1}^k\prod_{i=1}^k\mu(\d x_i),
\qe
where the sum ranges over all pairwise distinct $k$-indices from $1$ to $N$ and $\E$ is the expectation with respect to $\P$. Note that assuming that \eqref{DPPdistr} defines a probability measure and \eqref{def k cor} yield together that $K(x,y)$ has to be positive definite, namely  $\det[K(x_i,x_j)]_{i,j=1}^k\geq 0$ for every $k\geq 1$ and every $x_1,\ldots,x_k\in\Lambda$.

\paragraph{Projections and OP ensembles.} Note that the integral operator $K$ acting on $L^2(\mu)$ by
\eq
\label{op}
K:f(x)\mapsto\int K(x,y)\overline{f(y)}\mu(\d y)
\qe
is a bounded projection  onto $\mathrm{Im}(K)=\mathrm{Vect}(P_{0},\ldots,P_{{N-1}})$, the polynomials of degree at most $N-1$, parallel to $\mathrm{Ker}(K)^\perp=\mathrm{Vect}(Q_{0},\ldots,Q_{{N-1}})$. This projection is orthogonal if and only if $K^*=K$, which is equivalent to $P_k=Q_k$ for every $k\geq 0$. In this case, $P_k$ is the orthonormal polynomial associated with the reference measure $\mu$, and we say $\P$ is an \emph{orthogonal polynomial (OP) ensemble}.


\paragraph{Asymptotic.} From now, we consider a sequence of polynomial ensembles: For any $N\geq 1$ let  $\P_{N}$ be a polynomial ensemble with reference measure $\mu_{N}$ and kernel $K_{N}(x,y)$  associated to functions $P_{k}^{N}$ and $Q_{k}^{N}$ satisfying the same conditions as above. When the reference measure does not depend on $N$, we just write $P_{k}$ and $Q_{k}$. The goal is to study the behavior of the polynomial ensemble $\P_N$ in the large $N$ limit. The main tool we use here to do so are the recurrence coefficients.

\subsection{Recurrence coefficients}
\label{reccoef_sec}
Since $xP_{k}^N$ is a polynomial of degree $k+1$, the biorthogonality relations \eqref{biortho} yield, 
\eq
\label{rec def}
xP_{k}^N(x)=\sum_{m=0}^{k+1}\langle xP_{k}^N,Q_m^N\rangle P_{m}^N(x),
\qe
where $\langle \,\cdot,\cdot\,\rangle:=\langle \,\cdot,\cdot\,\rangle_{L^2(\mu_N)}.$ We refer to the coefficients $\langle xP_{k}^N,Q_m^N\rangle$ as \emph{recurrence coefficients}, since one can deduce inductively any $P_{k}^N$ from $P_0^N$ and these coefficients. 

\paragraph{Three-term recurrence relation.} 
In the setting of real OP ensembles,  where $P_k^N=Q_k^N$ are orthonormal polynomials with respect to  a measure $\mu_N$ on $\R$, we set as usual $$a_{k}^N:=\langle xP_{k}^N,P_{k+1}^N\rangle,\qquad b_k^N:=\langle xP_{k}^N,P_k^N\rangle,$$ with the convention that $a_{-1}^N:=0$, so as to recover  the three-term recurrence relation
\eq
\label{3 term rec}
xP_{k}^N(x)=a_{k}^N P_{k+1}^N(x)+b_{k}^N P_{k}^N(x)+a_{k-1}^N P_{k-1}^N(x).
\qe
When the reference measure does not depend on $N$, we just write $a_k$ and $b_k$.

\paragraph{The key formula.} The main message is that knowing the large $N$ limit of recurrence coefficients $\langle xP_{k}^N,Q_m^N\rangle$ as $N\to\infty$ provides a lot of information on the asymptotic behavior of polynomial ensembles.  Moreover, many of these asymptotic results turn out to a have quite simple combinatorial proofs, which may contrast with the usually quite involved analytical proofs in the field.

We will repeatedly use the following key observation: If one wants to express $x^\ell P_{k}^N$ in terms of the recurrence coefficients, it is convenient to consider the oriented graph $ G= (V,E)$ with vertices  $ V:=\N^2$ and  edges
\[
 E :=\Big\{(n,k)\rightarrow (n+1,m), \qquad  n,k\in\N,\quad 0\leq m \leq k+ 1\Big\},
\] 
where to each edge  is associated the weight 
\[
w\Big((n,k)\rightarrow (n+1,m)\Big):= \langle xP_{k}^N , Q_{m}^N \rangle.
\]
Indeed, the recurrence equation \eqref{rec def} yields by induction
$$
x^\ell P_{k}^N(x)=\sum_{m=0}^{k+\ell}\left(\sum_{\gamma : (0,k)\rightarrow (\ell,m)}\prod_{e\in\gamma}w(e)\right)P_{m}^N(x),
$$
where the rightmost sum ranges over the oriented paths $\gamma$ on $G$ starting from $(0,k)$ and ending at $(\ell,m)$, and each path picks the product of the weights along the edges it crosses. This leads to the  formula
\eq
\label{key}
\langle x^\ell P_{k}^N , Q_{m}^N\rangle=\sum_{\gamma : (0,k)\rightarrow (\ell,m)}\prod_{e\in\gamma}w(e),\qquad \ell,k,m\in\N,
\qe
which will be the key to study the large $N$ limit of the moments of polynomial ensembles. 

For example, in the case of a real OP ensemble, the paths in the rightmost sum \eqref{key} start from $(0,k)$, end after $\ell$ steps on the graph at $(\ell,m)$, and at each step the path increases its abscissa by one and  its ordinate by $1$, $0$, or $-1$, with corresponding weights given by 
\begin{align*}
w\Big((n,m)\rightarrow (n+1,m+1)\Big)& =a_{m+1}^N,\\
w\Big((n,m)\rightarrow (n+1,m)\Big)& =b_{m}^N,\\
w\Big((n,m)\rightarrow (n+1,m-1)\Big)& =a_{m}^N.
\end{align*}

\subsection{Organisation} In Section \ref{sec:mean} we investigate the convergence of the mean empirical distribution of polynomial ensembles. After recalling a few definitions, we provide in Theorem \ref{meanconv1} explicit limits for real OP ensembles having recurrence coefficients with continuous directional limits. This recovers and extends results obtained by \cite{Led04,Led05}. We partially extends this result to polynomial ensembles satisfying a finite-term recurrence relation in Proposition~\ref{finiterecTh}. 

Following \citep{Har15}, we study in Section \ref{sec:polchar} the mutual convergence of the mean distribution of polynomial ensembles and the zeros of associated average characteristic polynomials. We  discuss  further the complex setting and show in Corollary \ref{balayage} the two limiting measures have the same logarithmic energy away from the supports. 

In Section \ref{sec:var},  we provide upper bounds, comparison results, and limiting formulas for the variance of linear statistics. For real OP ensembles, after recalling the standard upper bound and explaining how this upgrades the mean convergence of the empirical measure to the almost sure one, we show how to obtain explicit formulas for the limiting variance in a rather broad setting, a proof inspired from \citep{BaHa16}. We also provide upper bounds and comparison estimates for moments of polynomial ensembles by using the approach of \citep{Har15}.

In Section \ref{sec:fluct}, we briefly present the recent results of \cite{BrDu13} and \cite{Lam15b} on fluctuations for the linear statistics of polynomial ensembles satisfying a finite-term recurrence relation.  

In Section \ref{sec:sim}, we review the standard algorithm for sampling DPPs associated with orthogonal projections due to \cite{HoKrPeVi06}.  Since this algorithm does not cover polynomial ensembles, except for OP ensembles, we suggest an alternative approach for this algorithm so as to cover more general DPPs associated with non-orthogonal projections, and in particular arbitrary polynomial ensembles.

\subsection*{Acknowledgements} 
This work has been written for the special issue of Constructive Approximation on the theme ``Approximation and statistical physics'' related to the workshop ``Optimal and random point configurations'' which took place at the Institut Henri Poincar\'e in June 2016. I would like to thank the organizers for giving me the opportunity to present the work \citep{BaHa16} there. I also would like to thank R\'emi Bardenet and Thomas Bloom for enriching discussions related to the present article. I acknowledge the support from CNRS through PEPS JCJC \textsc{DppMc} and from ANR through the grant ANR JCJC \textsc{BoB} (ANR-16-CE23-0003) and Labex \textsc{CEMPI} (ANR-11-LABX-0007-01).

\section{Mean global convergence}
\label{sec:mean}
We consider a sequence of polynomial ensembles $\P_N$ of $N$ particles $x_1,\ldots,x_N$ and investigate the limiting behavior of the empirical distribution
\eq
\label{emp}
\hat\mu_N:=\frac1N\sum_{i=1}^N\delta_{x_i}
\qe
under the law $\P_N$ as $N\to\infty$.  We start with a few definitions and standard facts on convergence of (random) probability measures. 

\subsection{Convergence of measures} Let $\mathcal P(\R^d)$ be the space of Borel probability measures on $\R^d$. A sequence $(\mu_n)$ in $\mathcal P(\R^d)$ \emph{converges weakly} to  a limiting measure $\mu$  when
\eq
\label{weak}
\int f \,\d\mu_n\xrightarrow[n\to\infty]{} \int f \,\d\mu
\qe
for every continuous and bounded functions $f$ on $\R^d$. By approximation this is equivalent to \eqref{weak} for every Lipschitz and bounded functions. When \eqref{weak} holds true for any polynomial function $f$, we say $\mu_n$ \emph{converges to $\mu$ in moments}. If $\mu$ is characterized by its moments, which is for instance the case when $\mu$ has compact support, then convergence in moments implies weak convergence but it is a stronger mode of convergence.

If $\mu_n$ is itself a random probability measure, namely a measurable map from a probability space $(\Omega_n,\mathscr F_n,\P_n)$ to $\mathcal P(\R^d)$, the mean of $\mu_n$ is the probability measure $\E\mu_n$ defined by $\int f\,\d\E\mu_n:=\E\int f\,\d\mu_n.$ Moreover, we say that $\mu_n$ converge in moments, or weakly, to $\mu$ \emph{almost surely} if for every probability space $(\Omega,\mathscr F,\P)$ such that $\P|_{\mathscr F_n}=\P_n$ we have $\P(\mu_n\to\mu \mbox{ in moments})=1$, or 
$\P(\mu_n\to\mu \mbox{ weakly})=1.$ \\

It follows from \eqref{def k cor} with $k=1$ that, with $\hat\mu_N$ defined in \eqref{emp},
\eq
\label{hatdef}
\E\hat\mu_N=\frac1N K_N(x,x)\mu_N(\d x)=\frac1N \sum_{k=0}^{N-1}P_k^N(x)\overline{Q_k^N}(x)\mu_N(\d x).
\qe
We now study the asymptotics of this mean measure. Results on the almost sure convergence of $\hat\mu_N$ will appear in Section \ref{sec:var}. 

\subsection{Real OP ensembles}
We  first show how the key formula  \eqref{key} easily yields an explicit description for the limit of the mean distribution $\E\hat\mu_N$ for the class of real OP ensembles having recurrence coefficients with continuous directional limits. 

The \emph{equilibrium measure of the interval} $[\alpha,\beta]$ is defined by
\eq
\label{arcsine}
\omega_{[\alpha,\beta]}(\d x):=
\displaystyle\frac1\pi\frac{\bv 1_{[\alpha,\beta]}(x) \d x}{\sqrt{(\beta-x)(x-\alpha)}}.
\qe
It is the unique minimizer of $\mu\mapsto \iint \log|x-y|^{-1}\mu(\d x)\mu(\d y)$ over Borel probability measures with support in $[\alpha,\beta]$. A random variable with distribution $\omega_{[\alpha,\beta]}$ is called an \emph{arcsine random variable on} $[\alpha,\beta]$. Our first result is the following.

\begin{theorem} 
\label{meanconv1} Let $\P_N$ be a sequence of real OP ensembles. Assume the recurrence coefficients satisfy  $a_k^N\to a(s)$ and $b_{k}^N\to b(s)$ as $k/N\to s\in(0,1)$, where $a,b:(0,1)\to\R$ are two continuous functions such that $a^kb^m$ are Riemann integrable on $[0,1]$ for every $k,m\geq 0$. Then we have the convergence in moments
$$
\E\hat\mu_N\xrightarrow[N\to\infty]{} \mu_{a,b}
$$
with $\mu_{a,b}$  the law of the random variable $2a(U)\xi+b(U)$, where $U,\xi$ are independent random variables with $U$ uniform on $[0,1]$ and $\xi$ is an arcsine random variable on $[-1,1]$. 
\end{theorem}

\begin{remark} If $V$ is uniform on $[0,1]$, then $\cos(\pi V)$ is an arcsine random variable on $[-1,1]$.  The moments of $\omega_{[-1,1]}$ have an explicit formula (make the change of variable $x=\cos\theta$ and do the Wallis' integrals recursion trick): 
\eq
\label{arcsine moments}
\int x^\ell\, \omega_{[-1,1]}(\d x)= 
\begin{cases}
\displaystyle\frac1{4^{m}}\binom{2m}{m} & \mbox{if } \ell=2m,\\
0 & \mbox{otherwise}.
\end{cases}
\qe
\end{remark}

\begin{corollary} If $\P_N$ be a sequence of real OP ensembles with a reference measure $\mu_N=\mu$ which does not depend on $N$ and $a_k\to a$ and $b_k\to b$ as $k\to\infty$, then 
$$
\E\hat\mu_N\xrightarrow[N\to\infty]{} \omega_{[-2a+b,2a+b]}
$$
in moments and weakly. 
\end{corollary}

\begin{remark} Denisov--Rakhmonov's Theorem, see e.g. \cite[Theorem 1.4.2]{Sim11}, states that if $\mu=\omega(x)\d x+\mu_s$ with $\mu_s$ singular and $\{x\in\R:\, \omega(x)>0\}=[-2a+b,2a+b]$ up to a set of null Lebesgue measure, then  $a_k\to a$ and $b_k\to b$ as $k\to\infty$. \end{remark}

\begin{proof}[Proof of Theorem \ref{meanconv1}] It follows from \eqref{hatdef} that, 
$$
\E\left[\,\int x^\ell\,\d\hat\mu_N\right]=\frac1{N}\sum_{k=0}^{N-1}\langle x^\ell P_{k}^N , P_{k}^N\rangle.
$$
In the case of real OPs, the paths in the key formula \eqref{key} for $\langle x^\ell P_{k}^N , P_{k}^N\rangle$ explore vertices contained in the set
$$
\Big\{ (n,m)\in\N^2 : \quad 0 \leq n\leq \ell,\quad k-\ell \leq m < k+\ell \Big\}.
$$
This yields  in particular that $\langle x^\ell P_{k}^N , P_{k}^N\rangle$ is a polynomial map in the variables $(a^N_{N+m})_{|m|\leq \ell}$ and $(b^N_{N+m})_{|m|\leq \ell}$ which does not depend on $k$ nor $N$. Namely, there exists a polynomial $\mathscr P_\ell$ such that
$$
\langle x^\ell P_{k}^N , P_{k}^N\rangle=\mathscr P_\ell(a_{k-\ell}^N,\ldots,a_{k+\ell}^N, b_{k-\ell}^N,\ldots,b_{k+\ell}^N)
$$
Moreover, given any $a,b\in\R$, we have
$$
\mathscr P_\ell(a,\ldots,a,b,\ldots,b)=\sum_{m=0}^{\lfloor \ell/2\rfloor} \binom{\ell}{2m}\binom{2m}{m} a^{2m}b^{\ell-2m},
$$
the first binomial term picking the  $2m$ steps ``up or down'' and the second the $m$ steps ``up'' within the previous non-flat steps. 
Using that $a_{k}^N\to a(s)$ and $b_{k}^N\to b(s)$ as $k/N\to s\in (0,1)$ and the assumptions on $a,b$, it follows that,
\begin{align*}
\frac1{N}\sum_{k=0}^{N-1}\langle x^\ell P_{k,N} , P_{k,N}\rangle & = \frac1{N}\sum_{k=0}^{N-1}\mathscr P_\ell(a_{k-\ell}^N,\ldots,a_{k+\ell}^N, b_{k-\ell}^N,\ldots,b_{k+\ell}^N)\\
&= \frac1{N}\sum_{k=0}^{N-1}\mathscr P_\ell(a(\tfrac kN),\ldots,a(\tfrac kN), b(\tfrac kN),\ldots,b(\tfrac kN)) +o(1)\\
& = \sum_{m=0}^{\lfloor \ell/2\rfloor} \binom{\ell}{2m}\binom{2m}{m} \frac1{N}\sum_{k=0}^{N-1}a(\tfrac kN)^{2m}b(\tfrac kN)^{\ell-2m}+o(1)\\
& =  \sum_{m=0}^{\lfloor \ell/2\rfloor} \binom{\ell}{2m}\binom{2m}{m} \int_0^1a(s)^{2m}b(s)^{\ell-2m}\d s +o(1),
\end{align*}
as $N\to\infty$. On the other hand, using the expression \eqref{arcsine moments} for $\E[\xi^k]$, we have 
\begin{align*} 
\E\Big[\big(2a(U)\xi+b(U)\big)^\ell\Big] & =\sum_{m=0}^{\ell}\binom{\ell}{m}\E\big[(2a(U)\xi)^m b(U)^{\ell-m}\big]\\
&=\sum_{m=0}^{\ell}\binom{\ell}{m}\int_0^1 \big(2a(s)\big)^m\E[\xi^m]b(s)^{\ell-m}\d s\\
& = \sum_{m=0}^{\lfloor \ell/2\rfloor}\binom{\ell}{2m}\binom{2m}{m}\int_0^1 a(s)^{2m}b(s)^{\ell-2m}\d s,
\end{align*}
and the result follows.
\end{proof}

\begin{example} 
\label{GUE}  The eigenvalues of a $N\times N$ GUE random matrix is an OP ensemble with $\mu_N=\exp(-Nx^2/2)\d x$ on $\R$. The recurrence coefficients are given by $a_{k}^N=\sqrt{\tfrac kN}$ and $b_{k}^N=0$. The theorem thus applies with $a(s)=\sqrt{s}$ and $b(s)=0$ and, since
$$
\E\Big[\big(2\sqrt{U}\xi\big)^\ell\Big]=2^{\ell}\E[\xi^\ell]\int_0^1 (\sqrt{u})^{\ell}\d u=\bv 1_{\ell=2m}\frac{1}{m+1}\binom{2m}{m}=\frac1{2\pi}\int_{-2}^2 x^\ell \sqrt{4-x^2}\ \d x,
$$
we recover the convergence towards the semi-circle law (weakly and in moments). 
\end{example} 

The convergence towards the law of $2\sqrt{U}\xi$ for the GUE has been obtained by \cite{Led04}, thanks to a Markov operator approach. It relies on the differential equations satisfied by the OPs  and applies to several classical OPs (Hermite, Laguerre, Jacobi). \cite{Led05} extends this approach so as to cover discrete OPs (Charlier, Meixner, Krawtchouk, Hahn).  It turns out that the recurrence coefficients of these classical continuous and discrete OPs are explicit and satisfy the conditions of Theorem \ref{meanconv1} after appropriate scalings, and Theorem \ref{meanconv1} recovers Ledoux's results.

\subsection{Polynomial ensembles}

In the general setting of polynomial ensembles, the three-term recurrence relations does not hold anymore, but in some examples finite-term recurrence relations do. By following the exact same lines of argument as in the proof of Theorem \ref{meanconv1} one can show this result:

\begin{proposition} 
\label{finiterecTh}
Let $\P_N$ be a sequence of polynomial ensembles and assume that:
\begin{itemize}
\item[{\rm(a)}] There exists $q\in\N$ independent on $k,N$ such that $\langle xP_{k}^N,Q_{k-m}^N\rangle=0$ for every $m>q$.
\item[{\rm(b)}] For every $-1\leq j\leq q$, there exists a continuous function $a_j:(0,1)\to\R$ such that 
$$\langle xP_{k}^N,Q_{k-j}^N\rangle\longrightarrow a_j(s),\qquad k/N\to s\in(0,1).$$
\item[{\rm(c)}] For every $k_{-1},\ldots,k_q\in\N$, $\prod_j a_j^{k_j}$ is Riemann integrable on $[0,1]$. 
\end{itemize}
Then, for any $\ell\geq 1$, 
\eq
\label{birotho}
\lim_{N\to\infty}\E\left[\,\int x^\ell\, \d\hat\mu_N\right]=\sum_{\bs k\in \mathscr D_\ell^{(q)}}\left(\begin{matrix}\ell \\ k_{-1},k_0,\ldots,k_q\end{matrix}\right)\int_0^1\prod_{j=-1}^q a_j^{k_j}(s)\d s,
\qe
where
$$
\mathscr D^{(q)}_\ell=\left\{  
\bs k=(k_{-1},k_0,k_1,\ldots,k_q)\in\N^{q+2}:\;\sum_{j=-1}^q k_j=\ell, \;\sum_{j=-1}^q  j k_j=0
\right\}.
$$
\end{proposition}

\begin{OP} For real polynomial ensembles, can we find a similar, or alternative, representation as in Theorem \ref{meanconv1} for the probability measure on $\R$ having for $\ell$-th moment the right hand side of \eqref{birotho}?
\end{OP}


\section{Zeros of average characteristic polynomials} 
\label{sec:polchar}
To a polynomial ensemble $\P_N$, one can associate the \emph{average characteristic polynomial}, 
\eq
\chi_N(z):=\mathbb E\left[\,\prod_{i=1}^N(z-x_i)\right],\qquad z\in\C,
\qe
where the expectation $\E$ refers to $\P_N$. For  OP ensembles, a formula attributed to Heine yields that $\chi_N$ is the $N$-th monic orthogonal polynomial. For several other polynomial ensembles $\chi_N$ have a similar connexion to the leading polynomial $P_N^N$, see e.g. \citep[Section 1]{Har15} and references therein.

Let $z_1,\ldots,z_N\in\C$ be the zeros of $\chi_N$. The next result, extracted from \citep{Har15}, states that these zeros have  the same moments than $\E\hat\mu_N$ asymptotically provided the recurrence coefficients do not grow too fast.

\begin{theorem}
\label{mainth} Let $\P_N$ be a sequence of polynomial ensembles satisfying,  for some $\ell>0$,
\eq
\label{sequence}
\max_{-\ell \leq k,m \leq \ell}\left|\langle  xP_{N+k}^N,Q_{N+m}^N\rangle\right| = o(N^{1/\ell}),\qquad N\rightarrow\infty.
\qe
Then, for every polynomial $P$ of degree $\mathrm{deg}(P)\leq \ell$,
\eq
\label{mainthcsq}
\lim_{N\rightarrow\infty}\left|\E\Bigg[\int P\,\d\hat\mu_N\Bigg] -\frac1N\sum_{i=1}^NP(z_i)\right|=0.
\qe
\end{theorem}

The pioneering result of this type is due to \cite{Sim09} and applies to OP ensembles with reference measure $\mu_N$ independent on $N$  with compact support. In this case \eqref{sequence} is automatically satisfied for every $\ell\geq  1$.

Since it is usually harder to derive the asymptotic distribution of the zeros of $\chi_N$ than the one of $\E\hat\mu_N=\frac1NK_N(x,x)\mu_N(\d x)$, this result may yield a simpler approach for obtaining weak limits for the zeros. For example, combined with Theorem \ref{meanconv1}, one recovers from Theorem \ref{mainth} the result of \cite{KuVA99} on weak limits of zeros of orthogonal polynomials.  Similarly, using  results from free probability, one can characterize in full generality the limiting zeros distribution of type II multiple Hermite and Laguerre polynomials in terms of free convolutions and obtain explicit algebraic equations for the Cauchy transform of these limiting distributions, see \cite[Section 3]{Har15}.

Note that in the complex setting the convergence \eqref{mainthcsq} for every polynomial in $z$ does not yield that the limiting zero distribution coincide with the one of $\E\hat\mu_N$. For example, consider the OP ensemble on the unit circle $\mathcal S^1$ with reference measure $\mu_N$  the uniform measure on $\mathcal S^1$, which corresponds to the eigenvalues of a $N\times N$ Haar unitary matrix.  Then $\chi_N(z)=z^N$ whereas $\E\hat\mu_N$ is the uniform measure on $\mathcal S^1$. However both measures have the same logarithmic potential outside of the unit disc. We show in the next corollary that this feature holds true in the general setting of polynomial ensembles.


\begin{corollary} 
\label{balayage}
Assume that  $\E\hat\mu_N$ and $\frac1N\sum\delta_{z_j}$ have subsequences which converge weakly towards $\mu$ and $\nu$ respectively, and that $K:=\Supp(\mu)\cup\Supp(\nu)$ is compact. Then, we have
$$
\int \log\frac{1}{|z-x|}\,\mu(\d x)=\int \log\frac{1}{|z-x|}\,\nu(\d x),\qquad z\in \C\setminus K^*,
$$
where $K^*$ is the smallest (for inclusion) simply connected subset of $\C$ which contains $K$.
\end{corollary}

\begin{proof} Using that $K$ is compact and Theorem \ref{mainth}, it follows that
$$
\int P\,\d\mu=\int P\,\d\nu
$$
for every polynomial $P$ and that these integrals are finite. Because the same holds true after taking the complex conjugate,  we further have 
\eq
\label{re P}
\int \Re P\,\d\mu=\int \Re P\,\d\nu.
\qe
Let $z\in\C\setminus K^*$ be fixed. Since $K^*$ is simply connected, there exists a determination $\phi$ of the map $x\mapsto \log(x-z)$ which is analytic on $K^*$. Since $K^*$ is bounded, there is a sequence of polynomials $P_n$ which approximates uniformly $\phi$ there. By taking the real part, this yields,
$$
\lim_{n\to\infty}\sup_{x\in K^*}\big|\Re P_n(x)-\log{|z-x|}\big|=0.
$$
Combined together with  \eqref{re P}, the corollary follows.

\end{proof}

\begin{OP} Under the assumptions of Corollary \ref{balayage}, can we obtain further information on the relation between $\mu$ and $\nu$? For instance, can we relate the logarithmic potentials of $\mu$ and $\nu$ inside $K^*$ by, say, an inequality?  
\end{OP}

We now turn to the proof of Theorem \ref{mainth}. The key idea is to notice that, using \eqref{def k cor}, one can write  $\chi_N$ as a Fredholm determinant 
\eq
\label{freddet}
\chi_N(z)=\det\big(z-K_N M K_N\big)_{L^2(\mu_N)}\, ,
\qe
where $M :f(x)\mapsto xf(x)$ is the position operator acting on $L^2(\mu_N)$. We refer to \cite[Proposition 2.3]{Har15} for a proof. Note that the operator $K_N M K_N$, seen as an endomorphism of $\mathrm{Im}(K_N)$, has the matrix representation $[\langle x P_{i-1}^N, Q_{j-1}^N\rangle ]_{i,j=1}^N$ in the basis $P_0^N,\ldots,P_{N-1}^N$. Hence in the OP ensemble setting this is just the $N\times N$ Jacobi matrix.

\begin{proof} [Proof of Theorem \ref{mainth}] Assume \eqref{sequence} holds true for some $\ell\geq 1$. To prove the theorem, it is enough to obtain \eqref{mainthcsq} with $P(x)=x^\ell$. Thanks to of the representation \eqref{freddet}, we have
$$
\sum_{i=1}^N z_i^\ell=\Tr\big((K_N MK_N)^\ell\big) =  \sum_{k=0}^{N-1}\langle \,\underbrace{K_NM \cdots K_NM}_{\ell}P_{k}^N , Q_{k}^N\rangle,
$$
where we used that $K_N^2=K_N$. If we introduce 
\eq
\label{DN}
D_N=\Big\{(n,m)\in\N^2 : \; m \geq N\Big\},
\qe
then using the notation of Section \ref{reccoef_sec} we have
\eq
\underbrace{K_NM \cdots K_NM}_{\ell}P_{k}^N=\sum_{m=0}^{N-1}\left(\sum_{\gamma : (0,k)\rightarrow (\ell,m),\;\gamma\cap D_N=\emptyset}w(\gamma)\right)P_{m}^N,
\qe
and hence
$$
\sum_{i=1}^N z_i^\ell=\sum_{k=0}^{N-1}\sum_{\gamma : (0,k)\rightarrow (\ell,k),\;\gamma\cap D_N=\emptyset}w(\gamma).
$$
Since the key formula \eqref{key} and  \eqref{hatdef} moreover yields
\eq
\E\left[\,\int x^\ell\,\d\hat\mu_N\right]=\frac1N\sum_{k=0}^{N-1}\langle x^\ell P_{k}^N , Q_{k}^N\rangle=\frac1N\sum_{k=0}^{N-1}\sum_{\gamma : (0,k)\rightarrow (\ell,k)}\prod_{e\in\gamma}w(e),
\qe
we obtain
\eq
\label{mainth1}
\E\left[\,\int x^\ell\,\d\hat\mu_N\right]-\frac1N\sum_{i=1}^Nz_i^\ell=\frac{1}{N}\sum_{k=0}^{N-1}\sum_{\gamma : (0,k)\rightarrow (\ell,k),\; \gamma\cap D_N\neq\emptyset}w(\gamma).
\qe
Since by following an edge  of the graph $G$ one increases the ordinate by  at most one, the  rightmost sum of \eqref{mainth1} will bring  null contribution if $k$ is  less that $N-\ell$. Observe moreover that the vertices explored by any path $\gamma$ going from $(0,k)$ to $(\ell,k)$ for some $N-\ell\leq k\leq N-1$ such that $\gamma \cap D_N\neq \emptyset$ form a subset of 
\[
B:=\Big\{ (n,m)\in\N^2 : \quad 0 \leq n\leq \ell,\quad N-\ell \leq m < N+\ell \Big\}.
\]
Since $\mathrm{Card}(B)\leq(2\ell)^\ell$, we obtain from \eqref{mainth1} the rough upper bound 
$$
\label{upperb}
\left|\E\left[\,\int x^\ell\,\d\hat\mu_N\right] -\frac1N\sum_{i=1}^Nz_i^\ell\,\right|
\leq \frac{ \left(2\ell\right)^\ell}{N}\max_{N-\ell\leq k,m\leq N+\ell}\left| \langle xP_{k}^N , Q_{m}^N\rangle\right|^\ell,
$$
and the theorem is proven.
\end{proof}

\section{Variance asymptotics}
\label{sec:var}

In this section, we emphasize on how recurrence coefficients allow to control the variance of linear statistics of polynomials ensembles. More precisely, we study for test functions $f$ the variance
\begin{multline}
\label{var_start}
\Var\left[\,\sum_{i=1}^Nf(x_i)\right]  =\int f(x)^{2}K_N(x,x)\mu_N(\d x)\\- \iint f(x) f(y) K_N(x,y)K_N(y,x)\mu_N(\d x)\mu_N(\d y), 
\end{multline}
a formula which follows from \eqref{def k cor} with $k=1$ and $k=2$. One motivation to obtain bounds on the variance is to upgrade the  convergence of $\E\hat\mu_N$ to the almost sure convergence of $\hat\mu_N$.

We start with the simplest setting of real OP ensembles, where in this case the Christoffel--Darboux formula is available and states that
\eq
\label{LN CD}
(x-y)K_N(x,y)= a_N^{N}\big(P_N^{N}(x)P_{N-1}^{N}(y)-P_{N-1}^{N}(x)P_N^{N}(y)\big).
\qe 
It is a direct consequence of the three-term recurrence relation \eqref{3 term rec}. This formula already yields an efficient upper bound on the variance for Lipschitz test functions, as we learned from \citep{PaSh11}.  

\begin{lemma} 
\label{easy bound var}
Let $\P_N$ be a real OP ensemble. For any Lipschitz function $f:\R\to\R$,  we have
$$
\sqrt{\Var\Big[\,\sum_{i=1}^{N}f(x_{i})\Big]}\leq a_{N}^{N}\|f\|_{\mathrm{Lip}},\qquad \|f\|_{\mathrm{Lip}}:=\sup_{x\neq y}\left|\frac{f(x)-f(y)}{x-y}\right|.
$$ 
\end{lemma} 

\begin{proof} We start from the symmetrized representation of the variance
\eq
\label{varrep2}
\Var\left[\,\sum_{i=1}^{N}f(x_{i})\right] =\frac12\iint \big(f(x)-f(y)\big)^{2}K_{N}(x,y)^{2}\mu_{N}(\d x)\mu_{N}(\d y),
\qe
which follows from \eqref{var_start} and that $K_N$ is an orthogonal projection. 
The Christoffel--Darboux formula \eqref{LN CD} and the orthonormality relations then yield
\begin{align*}
\Var\left[\,\sum_{i=1}^{N}f(x_{i})\right]& =\frac12\iint \big(f(x)-f(y)\big)^{2}K_{N}(x,y)^{2}\mu_{N}(\d x)\mu_{N}(\d y)\\
& \leq \frac12\|f\|_{\mathrm{Lip}}^2\iint (x-y)^{2}K_{N}(x,y)^{2}\mu_{N}(\d x)\mu_{N}(\d y)\\
& = \frac12(a_{N}^{N}\|f\|_{\mathrm{Lip}})^2\iint \big(P_N^{N}(x)P_{N-1}^{N}(y)-P_{N-1}^{N}(x)P_N^{N}(y)\big)^{2}\mu_{N}(\d x)\mu_{N}(\d y)\\
& = (a_{N}^{N}\|f\|_{\mathrm{Lip}})^2.
\end{align*}
\end{proof}

\begin{corollary} 
\label{coras}Let $\P_N$ be a sequence of real OP ensembles such that $\frac1Na_N^N$ is square-summable sequence. If $\E\hat\mu_N\to\mu$ weakly, then $\hat\mu_N\to\mu$ weakly almost surely.
\end{corollary}
\begin{proof} Given any $\epsilon>0$ and any bounded Lipschitz function $f$,  by assumption there exists $N_0$  such that 
$$
\sup_{N\geq N_0}\left|\E\left[\,\int f\,\d\hat\mu_N\right] -\int f \,\d \mu\right|\leq\epsilon.
$$
The Chebyshev inequality and Lemma \ref{easy bound var} then yield that, for every $N\geq N_0$,
\begin{align*}
\P_N\left(\left|\int f\,\d\hat\mu_N -\int f \,\d \mu\right|\geq 2\epsilon\right)& \leq \P_N\left(\left|\int f\,\d\hat\mu_N -\E\left[\,\int f\,\d\hat\mu_N\right]\right|\geq \epsilon\right)\\
& \leq \frac1{\epsilon^2}\Var\left[\,\int f\,\d\hat\mu_N\right]\\
 &= \frac1{N^2\epsilon^2}\Var\left[\,\sum_{i=1}^{N}f(x_{i})\right]\\
&\leq \frac{(a_N^N)^2}{N^2\epsilon^2}.
\end{align*}
By assumption, the right hand side is a summable sequence and hence  Borel--Cantelli's lemma yields that, for any joint probability space $(\Omega,\mathscr F,\P)$ satisfying $\P|_{\mathscr B(\Lambda^N)}=\P_N$, where $\mathscr B(\Lambda^N)$ stands for the Borel sets of $\Lambda^N$, we have
$$
\P\left(\limsup_{N\to\infty}\left|\int f\,\d\hat\mu_N -\int f \,\d \mu\right|\geq 2\epsilon\right)=0.
$$
Since this holds true for any $\epsilon>0$ and any bounded Lipschitz function $f$, this yields $\P(\hat\mu_N\to\mu \mbox{ weakly})=1$ and hence the corollary.
\end{proof}

Besides upper bounds, one may also investigate precise limits of the variances. As we have seen in the proof of Lemma \ref{easy bound var}, the study of the variance is linked to the study of the probability measure on $\R\times\R$,
 \eq
\label{LN d=1}
Q_N(\d x,\d y):=\frac1{2(a_{N}^N)^{2}}(x-y)^2K_N(x,y)^2\mu_{N}(\d x)\mu_{N}(\d y).
\qe
Indeed, \eqref{varrep2} yields the identity 
\eq
\label{varrep3}
\Var\left[\,\sum_{i=1}^{N}f(x_{i})\right]=(a_N^N)^2\iint \left(\frac{f(x)-f(y)}{x-y}\right)^2Q_N(\d x,\d y).
\qe

\begin{theorem}
\label{weak var} 
Let $\P_N$ be a sequence of real OP ensemble.  If there exists $a>0$ and $b\in\R$ such that, for every  fixed $k\in\Z$,
$$
a_{N+k}^N\to a ,\qquad b_{N+k}^{N}\to b,\qquad \mbox{ as } N\to\infty,
$$
then $Q_N$ converge in moments  towards 
\eq
\label{L d=1}
Q(\d x,\d y):=\frac{4a^{2} -(x-b)(y-b)}{4\pi^{2}a^{2}\sqrt{4a^{2}-(x-b)^2}\sqrt{4a^2-(y-b)^{2}}}\,\bv 1_{[-2a+b,2a+b]^{2}}(x,y)\d x\d y.
\qe
\end{theorem}

The weak convergence $Q_N\to Q$ has been established for the GUE, see Example \ref{GUE}, by \cite{HaTh12}, using the differential equations satisfied by Hermite polynomials. This more general statement is inspired from of \cite[Lemma 4.7]{BaHa16} and only requires asymptotic information on the recurrence coefficients.

The next result is a direct consequence of Theorem \ref{weak var}  and the representation \eqref{varrep3}.

\begin{corollary} 
\label{asptvar}
Under the assumptions of Theorem \ref{weak var}, for every $\mathscr C^{1}$ function $f$ satisfying $f(x)=O(|x|)$ as $x\to\infty$ we have,
\eq
\label{limvar}
\lim_{N\to\infty}\Var\left[\,\sum_{i=1}^{N} f(x_{i})\right]=a^2\iint \left(\frac{f(x)-f(y)}{x-y}\right)^2Q(\d x,\d y).
\qe
\end{corollary}

To prove Proposition \ref{weak var}, we start with the following comparison result. 

\begin{lemma} 
\label{compare}
Let $\P_N$ and $\breve \P_N$ be two sequences of real OP ensembles with recurrence coefficients satisfying, as $N\to\infty$,
\eq
\label{assss}
a_{N+k}^N=\breve a_{N+k}^N+o(1),\qquad b_{N+k}^N=\breve b_{N+k}^N+o(1),
\qe
for every fixed $k\in\Z$. Then,  for every bivariate polynomial $P$, we have
$$
\iint P(x,y) \,Q_N(\d x,\d y)=\iint P(x,y) \,\breve Q_N(\d x,\d y)+o(1),\qquad \mbox{ as }N\to\infty.
$$
\end{lemma}

\begin{proof}  It is enough to prove that for every $m,n\in\N$, 
\[
\lim_{N\to\infty}\left|\iint x^m y^n Q_N(\d x,\d y)-\iint x^m y^n \breve Q_N(\d x,\d y)\right|=0.
\]
Since the Christoffel--Darboux formula \eqref{LN CD} yields
\eq
\label{CDQN}
Q_N(\d x,\d y)=\frac12\big(P_N^{N}(x)P_{N-1}^{N}(y)-P_{N-1}^{N}(x)P_N^{N}(y)\big)^2,
\qe
we see it is enough to show that, for every $m\in\N$,
\eq
\label{moment 1}
\lim_{N\to\infty}\left|\langle x^m P_N^N,P_N^N\rangle_{L^2(\mu_N)} -\langle x^m \breve P_N^N,\breve P_N^N\rangle_{L^2(\breve\mu_N)} \right| =0,
\qe
and 
\eq
\label{moment 2}
\lim_{N\to\infty}\left|\langle x^m P_N^N,P_{N-1}^N\rangle_{L^2(\mu_N)} -\langle x^m \breve P_N^N,\breve P_{N-1}^N\rangle_{L^2(\breve\mu_N)} \right| =0.
\qe
As it follows from the key formula \eqref{key} that $\langle x^m P_N^N,P_N^N\rangle_{L^2(\mu_N)}$ and $\langle x^m P_N^N,P_{N-1}^N\rangle_{L^2(\mu_N)}$ are both polynomial functions in the variables $(a^N_{N+k-1})_{-\ell\leq k\leq \ell+1}$ and $(b^N_{N+k-1})_{-\ell\leq k\leq \ell+1}$, polynomial functions which only depend on $m$. Thus \eqref{moment 1} and \eqref{moment 2} both follow from  the assumption \eqref{assss} and the lemma is proven. 

\end{proof}

\begin{proof}[Proof of Theorem \ref{weak var}] 

By a change of variables, without loss of generality one can assume $a=1/2$ and $b=0$. In this case, 
\eq
\label{Qtest}
Q(\d x,\d y)=\frac{1 -xy}{\pi^{2}\sqrt{1-x^2}\sqrt{1-y^{2}}}\bv 1_{[-1,1]^{2}}(x,y)\d x\d y.
\qe


We first prove the result when the reference measure $\mu_N$ equals the equilibrium measure $\omega_{[-1,1]}$ for every $N$. In this case the orthonormal polynomials $P_k$ are the Chebyshev polynomials, defined by $P_0=1$ and $P_k(\cos\theta)=\sqrt2 \cos(k\theta)$ when $k\geq 1$. The recurrence coefficients are given by $a_k=\bv 1_{k=0}1/\sqrt2+\bv 1_{k\geq 1}1/2$ and $b_k=0$. Let us call $Q_N^*$  the associated measure \eqref{LN d=1}  in this setting. Using the representation \eqref{CDQN}, we see the image of $Q_N^*$ by the change of variables $(x,y)=(\cos\theta,\cos\eta)$, where $\theta,\eta\in[0,\pi]$, reads
\eq
\label{transformed LN}
\frac2{\pi^2} \big(\cos(N\theta)\cos((N-1)\eta)- \cos((N-1)\theta\cos(N\eta)\big)^2\d\theta\d\eta.
\qe
This measure  has for Fourier transform
\begin{multline*}
\frac2{\pi^2} \int_0^\pi\int_0^\pi \mathrm{e}^{i(\theta u+ \eta v)}\big\{\cos( N\theta)\cos((N-1)\eta)- \cos((N-1)\theta)\cos(N\eta)\big\}^2\d\theta\d\eta\\
= \frac2{\pi^2} \int_0^\pi\int_0^\pi\cos(\theta u+ \eta v)\big\{\cos( N\theta)\cos((N-1)\eta)- \cos((N-1)\theta)\cos(N\eta)\big\}^2\d\theta\d\eta.
\end{multline*}
By developing the square in the integrand and linearizing the products of
cosines, we see that the non-vanishing contribution  as $N\to\infty$ of the
Fourier transform are the terms which are independent on $N$ since the $N$-dependent terms come up with a factor $1/N$ after integration. Thus, the Fourier transform equals to
$$
\frac1{\pi^2} \int_0^\pi\int_0^\pi\cos(\theta u+ \eta v)\big(1-\cos\theta\cos\eta\big)\d\theta\d\eta+\cO(1/N).
$$
This yields the weak convergence of \eqref{transformed LN} towards $\pi^{-2}(1-\cos\theta\cos\eta)\d\theta\d\eta$. By taking the image of the measures by the inverse map $(\cos\theta,\cos\eta)\mapsto(x,y)$, we obtain the weak convergence of $Q_N^*$ towards \eqref{Qtest}. Since all these measures are supported on the same compact set $[-1,1]^2$, by approximation this weak convergence extends to convergence in moments.

In the general setting of a reference measure $\mu_N$ such that $a^N_{N+k}\to1/2$ and $b^N_{N+k}\to 0$ as $N\to\infty$ for every fixed $k$, Lemma \ref{compare} yields that, for any bivariate polynomial $P$, 
$$
\left|\int P\, \d Q_N-\int P\, \d Q\right|\leq \left|\int P\, \d Q_N-\int P\, \d Q_N^*\right|+\left|\int P\, \d Q_N^*-\int P\, \d Q\right|\xrightarrow[N\to\infty]{}0,
$$
which proves the theorem.

\end{proof}

In the general setting of polynomial ensembles, the Christoffel--Darboux formula is not available anymore, but one can still obtain upper bounds on the variances, as well as a comparison result, thanks to a similar path representation as in the key formula \eqref{key}. 

\begin{theorem}
\label{variance_univ} 
\begin{itemize}
\item[{\rm(a)}]Let $\P_N$ be a polynomial ensemble.  For every $\ell\geq 1$, we have
\eq
\label{varupperb}
\Var\Bigg[\,\sum_{i=1}^Nx_i^\ell\Bigg] 
\leq (2\ell)^{2\ell}\max_{-\ell\leq k,m\leq \ell}\left| \langle xP_{N+k}^N , Q_{N+m}^N\rangle\right|^{2\ell}.
\qe
\item[{\rm(b)}]Let $\P_N$ and $\breve \P_N$ be two sequences of polynomial ensembles satisfying
$$
\langle xP_{N+k}^N,Q_{N+m}^N\rangle_{L^2(\mu_N)}=\langle x\breve P_{N+k}^N,\breve Q_{N+m}^N\rangle_{L^2(\breve\mu_N)}+o(1),\qquad N\to\infty,
$$
for every fixed $k,m\in\Z$. Then,  for every univariate polynomial $P$,
$$
\Var\left[\,\sum_{i=1}^NP(x_i)\right]=\breve\Var\left[\,\sum_{i=1}^NP(x_i)\right]+o(1),\qquad N\to\infty.
$$
\end{itemize}
\end{theorem}

Part (a) is taken from \citep{Har15}.

\begin{proof} We start from the representation \eqref{var_start}, 
\[
\Var\left[\,\sum_{i=1}^Nx_i^\ell\right]  =\int x^{2\ell}K_N(x,x)\mu_N(\d x)- \iint x^\ell y^\ell K_N(x,y)K_N(y,x)\mu_N(\d x)\mu_N(\d y).
\]
Using the position operator $Mf(x)=xf(x)$ this can be alternatively written as
\eq
\label{B}
\Var\left[\,\sum_{i=1}^Nx_i^\ell\right] =  {\rm Tr}\big(K_N M^{2\ell}K_N\big)-{\rm Tr}\big(K_N M^\ell K_N M^\ell K_N\big).
\qe
We have from the key formula \eqref{key},
\eq
\label{BB}
{\rm Tr}\big(K_N M^{2\ell}K_N\big)=\sum_{k=0}^{N-1}\langle x^{2\ell} P_{k}^N , Q_{k}^N\rangle=\sum_{k=0}^{N-1}\sum_{\gamma : (0,k)\rightarrow (2\ell,k)}w(\gamma).
\qe
Since 
\[
(K_NM^\ell K_NM^\ell K_N )P_{k}^N=\sum_{m=0}^{N-1}\left(\sum_{\gamma : (0,k)\rightarrow (\ell,m),\;\gamma(\ell)<N}w(\gamma)\right)P_{m}^N,
\]
where  $\gamma(\ell)$ stands for the ordinate of the path $\gamma$ at abscissa $\ell$,  we moreover obtain  
\[
{\rm Tr}\big(K_N M^\ell K_N M^\ell K_N\big)=\sum_{k=0}^{N-1}\sum_{\gamma : (0,k)\rightarrow (2\ell,k),\; \gamma(\ell)<N}w(\gamma).
\]
Combined with \eqref{B}--\eqref{BB} this yields
\eq
\label{pathvar}
\Var\left[\,\sum_{i=1}^Nx_i^\ell\right] =\sum_{k=0}^{N-1}\sum_{\gamma : (0,k)\rightarrow (2\ell,k),\; \gamma(\ell)\geq N}w(\gamma).
\qe
Since at each step a path can increase its ordinate by at most one, the condition $\gamma(\ell)\geq N$ yields the contributing paths in right hand side of \eqref{pathvar} having vertices lying within the set
\[
\Big\{ (n,m)\in\N^2 : \quad 0 \leq n\leq 2\ell,\quad N-\ell \leq m < N+\ell \Big\}.
\]
Thus we have the rough upper bound 
$$
\label{upperb}
\Var\left[\,\sum_{i=1}^Nx_i^\ell\right]
\leq { \left(2\ell\right)^\ell}\max_{N-\ell\leq k,m\leq N+\ell}\left| \langle xP_{k}^N , Q_{m}^N\rangle\right|^{2\ell},
$$
which proves (a). Moreover,  this shows \eqref{pathvar} is a polynomial function in the variables  $\langle xP^{N}_{N+k} , Q_{N+m}^N\rangle$ where $-\ell\leq k,m\leq\ell$, a polynomial function which only depends on $\ell$, and (b) follows.

\end{proof}

Now the same proof as for Corollary \ref{coras} yields the following upgrade for the moment convergence of real polynomial ensembles. 

\begin{corollary} Let $\P_N$ be a sequence of real polynomial ensembles with recurrence coefficients satisfying the growth assumption \eqref{sequence} for every $\ell\geq 1$. If $\E\hat\mu_N\to\mu$ in moments, then $\hat\mu_N\to\mu$ in moments almost surely.
\end{corollary}

\section{Higher order cumulants and fluctuations}
\label{sec:fluct}

After studying the variance asymptotics, it is natural to investigate higher order cumulants. We briefly review here recent results concerning this question and emphasize on the role played by recurrence coefficients. More precisely, the cumulants $\kappa_n[X]$ of a real random variable $X$ are defined through the series expansion of the log-Laplace transform,
$$
\log\E[\mathrm{e}^{zX}]=\sum_{n=1}^\infty \kappa_n[X] z^n,
$$
provided this makes sense. Thus we have $\kappa_1[X]=\E[X]$ and $\kappa_2[X]=\Var[X]$. A useful characterization of a real gaussian random variable $X$ is that all its cumulants vanish for $n\geq 3$. This is sometimes useful to prove convergence in law of a properly rescaled random variable towards a gaussian limit, such as in the classical central limit theorem (CLT). 

In the setting of real polynomial ensembles satisfying a finite-term recurrence relation, many interesting results for the cumulants of linear statistics $\sum f(x_i)$  have been obtained by \cite{BrDu13}. In particular they obtain central limit theorems for the fluctuations of the linear statistics when convergence of the recurrence coefficients is assumed. These results have been recovered by \cite{Lam15b} using technics more in the spirit of this note, namely involving sums over paths weighted by the recurrence coefficients. 

Specifically, one can extract from \citep{BrDu13} the following results, which should be put in perspective with Theorem \ref{variance_univ}(b) and Theorem \ref{weak var}: Let us say that a sequence of polynomial ensembles $\P_N$ is \emph{banded} if there exists $R\geq 1$ independent on $k,m,N$ such that $\langle xP_{k}^N,Q_{k-m}^N\rangle_{L^2(\mu_N)}=0$ for any $m> R$ and $k\geq 0$.

\begin{theorem} 
\label{BD}
Let $\P_N$ and $\breve \P_N$ be  real and banded polynomial ensembles satisfying
$$
\langle xP_{N+k}^N,Q_{N+m}^N\rangle_{L^2(\mu_N)}=\langle x\breve P_{N+k}^N,\breve Q_{N+m}^N\rangle_{L^2(\breve\mu_N)}+o(1),\qquad N\to\infty,
$$
for every fixed $k,m\in\Z$. Then,  for every $n\geq 2$ and every univariate polynomial $P$,
$$
\kappa_n\left[\,\sum_{i=1}^NP(x_i)\right]=\breve\kappa_n\left[\,\sum_{i=1}^NP(x_i)\right]+o(1),\qquad N\to\infty.
$$
Moreover, if $\langle xP_{N+k}^N,Q_{N+m}^N\rangle_{L^2(\mu_N)}$ has a limit as $N\to\infty$ for every $k,m\in\Z$, then
\eq
\label{limcum}
\lim_{N\to\infty}\kappa_n\left[\,\sum_{i=1}^NP(x_i)\right]=0,\qquad n\geq 3.
\qe
\end{theorem}

As explained above, \eqref{limcum} implies Gaussian fluctuations for the linear statistics $\sum P(x_i)$ once centered and reduced.  Combined with Theorem \ref{weak var} and Corollary \ref{asptvar} this  yields for instance a general CLT for real OP ensembles:

\begin{corollary}
Let $\P_N$ be a sequence of real OP ensemble.  If there exists $a>0$ and $b\in\R$ such that, for every  fixed $k\in\Z$,
$$
a_{N+k}^N\to a ,\qquad b_{N+k}^{N}\to b,\qquad \mbox{ as } N\to\infty,
$$
then, for any polynomial $P$, we have the convergence in law to a Gaussian random variable,
\[
\sum_{i=1}^NP(x_i)-\E\left[\,\sum_{i=1}^NP(x_i)\right] \xrightarrow[N\to\infty]{*}\mathcal N(0,\sigma_P^2)\,,
\]
where the limiting variance is given by
\eq
\label{limvarOP}
\sigma_P^2:=a^2\int_{-2}^2\int_{-2}^2 \left(\frac{f(ax+b)-f(ay+b)}{x-y}\right)^2 \frac{4-xy}{\sqrt{4-x^2}\sqrt{4-y^2}}\,\d x\d y.
\qe
\end{corollary}
One recovers the CLT known for the GUE random matrices where $a=1$ and $b=0$.

In \citep{BrDu13}, all these results are extended from polynomials $P$ to $\mathscr{C}^1$ functions $f$ satisfying an appropriate  growth condition, by means of a density argument. They also obtain CLTs for more general banded polynomial ensembles, although the limiting variance has a less explicit form than \eqref{limvarOP}, which is a weighted version of  the Sobolev norm $H^{1/2}$. \\

A similar CLT for the linear statistics of higher dimensional DPPs associated with projection onto multivariate orthogonal polynomials has  been derived in \citep{BaHa16}, with a proof strongly using recurrence coefficients' asymptotics. These are DPPs generating $N$ points on $\R^d$, for any $d\geq 1$, which are the higher dimensional analogues of the OP ensembles; we introduced such processes for the purpose of building a Monte Carlo method converging faster than the classical methods, based on weakly or non correlated random variables. The main difference in higher dimension is that now the variance  grows with the number of points $N$. More precisely it is proved that, as $N\to\infty$,
\eq
\label{varR^d}
\Var\left[\,\sum_{i=1}^{N}f(x_{i})\right]\sim N^{1-1/d} \,\sigma_f^2\,,
\qe
where $\sigma_f^2$ is an explicit constant depending on the $\mathscr C^1$ test function $f:\R^d\to\R$. Due to a result of Soshnikov this automatically yields a CLT for $\sum f(x_i)$, but the variance asymptotics \eqref{varR^d} is much harder to obtain than in the one dimensional setting we describe above.  Comparing to the case of $N$ i.i.d random variables $X_i$ on $\R^d$ where 
$$
\Var\left[\,\sum_{i=1}^{N} f(X_i)\right]\sim N \,\tilde \sigma_f^2\,,
$$
these higher dimensional point processes are sometimes called \emph{hyperuniform}.

\begin{OP} Given DPPs associated with finite rank $N$ projections $K_N$ over functions acting on an arbitrary dimensional space, which  exponents $\alpha$ for  the variance's growth
\[
\Var\left[\,\sum_{i=1}^{N}f(x_{i})\right]\sim N^{\alpha} \,\sigma_f^2\,
\]
are possible for smooth test function $f$, assuming we are in the global regime? The latter means that  $\frac1NK_N(x,x)\mu_N(\d x)$ has a non-trivial weak limit when $N\to\infty$, since otherwise one can always modify the exponent $\alpha$ by scaling the particle system. It appears that $\alpha$ depends more on the space of functions than the real dimension of the ambiant space itself. For instance, a variance growth of $N^{1-1/d}$ is achived in \citep[Theorem 1.5]{Ber16} for  DPPs living on a complex manifold of dimension $d$, thus locally diffeomorphic to $\R^{2d}$.
\end{OP}

%
%
%
%
%
%
%
%

\section{Sampling a DPP with non-orthogonal projection kernel}
\label{sec:sim}
In this section, we consider the general setting where $K(x,y)$ is a  complex valued positive definite kernel associated with a rank $N$ projection acting $K$ on $L^2(\mu)$, for some appropriate reference measure $\mu$ with support $\Lambda$. Namely, we only assume the kernel is positive definite and satisfies
\eq
\label{Kcond}
\int K(x,x)\mu(\d x)=N,\qquad \int K(x,u)K(u,y)\mu(\d u)=K(x,y),\qquad x,y\in \Lambda.
\qe
Under these assumptions, it is well known that the permutation invariant measure on $\Lambda^N$
\eq
\label{dP}
\d \P(x_1,\ldots,x_N)=\frac1{N!}\det\Big[K(x_i,x_j)\Big]_{i,j=1}^N\prod_{j=1}^N\mu(\d x_j)
\qe
is a probability distribution  and that the point process $x_1,\ldots,x_N$ generated by $\P$ is determinantal with kernel $K(x,y)$. The latter assertion means that, for every $k\geq  1$ and fixed $x_1,\ldots,x_k\in\Lambda$,
\eq
\label{rhok}
\frac{1}{(N-k)!}\int_{X^{N-k}} \d\P(x_1,\ldots,x_N)=\det\Big[K(x_i,x_j)\Big]_{i,j=1}^k.
\qe
 
The goal of this section is to provide a sampling algorithm for $x_1,\ldots,x_N$ with joint probability distribution \eqref{dP}. We first review the HKPV algorithm, introduced by  \cite{HoKrPeVi06}, in a self contained way. This algorithm is based on a neat geometric interpretation but only works when the kernel $K(x,y)$ is hermitian. Next, we reinterpret this algorithm so as to get rid of this hermitian assumption.

\subsection{The HKPV algorithm}
In this subsection, we make the extra assumption that $K$ is hermitian, i.e. $K(y,x)=\overline{K(x,y)}$. The idea behind the HKPV algorithm goes as follows: Using that $K$ is hermitian and the reproducing property \eqref{Kcond}, we obtain that
$$
K(x_i,x_j)=\int K(x_i, y)\overline{K(x_j,y)}\mu(\d y)=\langle \psi_i,\psi_j\rangle,\qquad \psi_j(x):=K(x_j,x),
$$
where $\langle\cdot,\cdot\rangle:=\langle\cdot,\cdot\rangle_{L^2(\mu)}$.
Hence,
\eq
\label{det=Gram}
\det\Big[K(x_i,x_j)\Big]_{i,j=1}^N=\det\Big[\langle \psi_i,\psi_j\rangle\Big]_{i,j=1}^N
\qe
is a Gram matrix which represents the squared volume of the parallelotope generated by the vectors $\psi_1,\ldots,\psi_N$. Note that by \eqref{dP}, \eqref{det=Gram} the  $\psi_k$'s are linearly independent for $\P$-almost every configuration $x_1,\ldots,x_N$.  Thus, using the generalized ``base  times height'' formula we can write this determinant as a product of $N$ functions. It turns out the $k$-th function only involves $x_1,\ldots,x_k$ and is a probability density with respect to $\mu$ in the variable $x_k$. This allows us to sample  $x_1,\ldots,x_N$ from $\P$ inductively by sampling $x_1$ according the the first density, and then $x_2$ according to the second density knowing $x_1$, etc. \\

More precisely, set $H_N:=\Span(\psi_1,\ldots,\psi_N)$ and orthogonalize  the family $\psi_k$ by setting $$\hat\psi_1:=\psi_1,\qquad \hat\psi_{k+1}:= P_{H_{N-k}}(\psi_{k+1}),$$ where $P_H$ stands for the orthogonal projection onto $H$ and $H_{N-k}$ is the orthocomplement of $\Span(\psi_{1},\ldots,\psi_k)$ in $H_N$. That is, we apply the Gram--Schmidt algorithm to the $\psi_k$'s except that we do not normalize at each step the resulting orthogonal family. The interest in doing so is the generalized ``base  times height'' formula:
\eq
\label{basexheight}
\det\Big[\langle \psi_i,\psi_j\rangle\Big]_{i,j=1}^N=\prod_{k=1}^N\|\hat\psi_k\|^2.
\qe
\begin{proof} If $\mathcal R$ is the endomorphism of $H_N$ defined by $\mathcal R\hat\psi_k=\psi_k$ for every $1\leq k\leq N$, then its matrix representation in the basis $(\hat\psi_k)$ reads  
$$
R=\left[\frac{\langle \psi_i,\hat\psi_j\rangle}{\|\hat\psi_j\|^2}\right]_{i,j=1}^N.
$$
Since  $\psi_k=\hat\psi_k+\phi$ with $\phi\in\Span(\hat\psi_1,\ldots,\hat\psi_{k-1})$ by construction,  $R$ is lower triangular  and moreover $R_{jj}=1$ for every $j$. Thus, $\det R=1$ and \eqref{basexheight} follows since,  using that $\hat\psi_k$ is an orthogonal family, we have
$$
\det\Big[\langle \psi_i,\psi_j\rangle\Big] =\det\Big[\langle \mathcal R\hat\psi_i,\mathcal R\hat\psi_j\rangle\Big]=\det\Big(R\Big[\langle \hat\psi_i,\hat\psi_j\rangle\Big]R^*\Big)=|\det R|^2\prod_{k=1}^N\|\hat\psi_k\|^2.
$$
\end{proof}

By combining \eqref{dP},\eqref{det=Gram} and \eqref{basexheight}, we thus obtain the identity 
\eq
\label{factorization}
\d \P(x_1,\ldots,x_N)=\frac1{N!}\prod_{k=1}^N\|\hat\psi_k\|^2\mu(\d x_k).
\qe
For the first term, we have the explicit formula
\eq
\label{eta1}
\eta_1(\d x_1):=\frac1N \|\hat\psi_1\|^2\mu(\d x_1)= \frac1N\|\psi_1\|^2\mu(\d x_1)=\frac1NK(x_1,x_1)\mu(\d x_1),
 \qe
 and we see that $\eta_1$ is a probability measure from \eqref{Kcond}.
 Then, for any $k\geq 1$,  
 $$
 \eta_{k+1}(\d x_{k+1}| x_1,\ldots, x_{k}):=\frac1{N-k}\|\hat\psi_{k+1}\|^2\mu(\d x_{k+1})
 $$ is non-negative and only depends on $x_1,\ldots,x_{k+1}$ by construction, and we have from \eqref{factorization},
\eq
\label{chain rule 1}
\d \P(x_1,\ldots,x_N)=\eta_{1}(\d x_{1})\prod_{k=2}^N\eta_k(\d x_k|x_1,\ldots,x_{k-1}).
\qe
  Moreover, $ \eta_{k+1}(\,\cdot\,| x_1,\ldots, x_{k})$  it is a probability measure on $\Lambda$. 
 
 \begin{proof}Given any subspace $H\subset H_N$, if we let $K_H$ be the orthogonal projection onto $H$ then clearly $K_HK=K_H$. If $K_H(x,y)$ stands for the kernel of $K_H$, the latter identity yields that $K_{H}\psi_{k}=K_H(x_{k},\cdot)$. As a consequence,
\begin{align*}
 \int \eta_{k+1}(\d x_{k+1}| x_1,\ldots, x_{k})&=\frac1{N-k}\int \|K_{H_{N-k}}\psi_{k+1}\|^2\mu(\d x_{k+1})\\
& =\frac1{N-k}\int \|K_{H_{N-k}}(x_{k+1},\cdot)\|^2\mu(\d x_{k+1})\\
&=\frac1{N-k}\int K_{H_{N-k}}(x_{k+1},x_{k+1})\mu(\d x_{k+1})\\
&=\frac1{N-k}\mathrm{Tr}(K_{N-k})=1.
\end{align*}
\end{proof}
In conclusion, we see from \eqref{chain rule 1} that sampling $x_1,\ldots,x_N$ with joint distribution $\P$ amounts to sample  sample $x_1$ with distribution $\eta_1$, then $x_2$ with distribution $\eta_2(\,\cdot\,|x_1)$, then  $x_3$ with distribution $\eta_3(\,\cdot\,|x_1,x_2)$, etc. This is the HKPV algorithm introduced in \citep{HoKrPeVi06}.

For this algorithm to be implementable in practice, one needs to be able to sample according to the distributions $\eta_k$'s. For $\eta_1$ we have the explicit  formula \eqref{eta1} and, assuming one knows $\mu$ and an explicit upper bound on $K(x,x)$, one can sample from this law by rejection sampling. As for the other densities, we have by orthogonality
\eq
\label{etak HKPV}
\frac{ \eta_{k+1}(\d x\,| x_1,\ldots,x_{k})}{\mu(\d x)}=\frac1{N-k}\|\hat\psi_k\|^2  =\frac1{N-k}\left(K(x,x)-\sum_{j=1}^{k-1}\frac{\langle\psi_k,\hat\psi_j\rangle^2}{\|\hat\psi_j\|^2}\right),
\qe
which allows pointwise evaluations inductively. See \cite[Section 4]{ScZaTo09} and \cite[Section 2.4]{LaMoRu15} for further details on how to implement the HKPV algorithm in practice.

\subsection{The HKPV algorithm revisited}
The previous geometric interpretation, starting point of the HKPV algorithm, becomes unclear when $K(x,y)$ is not hermitian. One may notice the previous algorithm amounts to perform a Cholesky decomposition  $[K(x_i,x_j)]_{i,j=1}^N=LDL^*$ and to store the entries of the diagonal matrix $D$. Performing a $LDU$ decomposition would extend the algorithm to the non-hermitian setting. This works, but we suggest the following simple approach instead.

First, \eqref{rhok} yields that for $\P$-almost every configuration $x_1,\ldots,x_N$ and every $1\leq k\leq N$, 
\eq
\label{posit} 
\det\Big[K(x_i,x_j)\Big]_{i,j=1}^{k}>0.
\qe
Consider the mean distribution, 
\eq
\eta_1(\d x_1):=\frac1NK(x_1,x_1)\mu(\d x_1),
\qe
 as well as the marginal distributions given, for $1\leq k \leq N-1$ and $x_1,\ldots,x_k$ satisfying \eqref{posit}, by
\eq
\label{etak}
\eta_{k+1}(\d x_{k+1}|x_1,\ldots,x_k):=\frac1{N-k}\frac{\det\big[K(x_i,x_j)\big]_{i,j=1}^{k+1}}{\det\big[K(x_i,x_j)\big]_{i,j=1}^{k}}\, \mu(\d x_{k+1}).
\qe
We define the density of $\eta_{k+1}(\d x_{k+1}|x_1,\ldots,x_k)$ to be zero otherwise. Clearly,
\eq
\label{chain rule 2}
\d \P(x_1,\ldots,x_N)=\eta_{1}(\d x_{1})\prod_{k=2}^N\eta_k(\d x_k|x_1,\ldots,x_{k-1}).
\qe
It follows from \eqref{Kcond} that $\eta_1$ is a probability measure and, because \eqref{rhok} yields 
$$
\int \det\Big[K(x_i,x_j)\Big]_{i,j=1}^{k+1}\mu(\d x_{k+1})=(N-k)\det\Big[K(x_i,x_j)\Big]_{i,j=1}^{k},
$$
so does $\eta_{k+1}(\,\cdot\, | x_1,\ldots,x_k)$ for every $1\leq k\leq N-1$ and $x_1,\ldots,x_k$ such that \eqref{posit} holds true.\\

Thus the same conclusion than in the HKPV algorithm applies: sampling $x_1,\ldots,x_N$ with joint distribution $\P$ amounts to sample  sample $x_1$ with distribution $\eta_1$, then $x_2$ with distribution $\eta_2(\,\cdot\,|x_1)$, etc. 
As before $\eta_1$ is explicit and, using a Schur complement in \eqref{etak}, we have for any $k\geq 2$ the explicit formulas for the densities:
\begin{multline}
\label{formule cool}
\frac{ \eta_{k+1}(\d x\,| x_1,\ldots,x_{k})}{\mu(\d x)}\\
=\frac1{N-k}\left(K(x,x)-\left[\begin{matrix} K(x,x_1)\\
\vdots\\
K(x,x_{k})\end{matrix}\right]^T\left(\Big[K(x_i,x_j)\Big]_{i,j=1}^{k}\right)^{-1}\left[\begin{matrix} K(x_1,x)\\
\vdots\\
K(x_{k},x)\end{matrix}\right]\right).
\end{multline}
Note that, by successive integrations, a decomposition of the form \eqref{chain rule 2} into probability measures  has to be unique, and hence both algorithms yield the same $\eta_k$'s when $K(x,y)$ is hermitian. Actually, one can work out the formula \eqref{etak HKPV} so as to obtain \eqref{formule cool}, as pointed out in \cite[Section  2.4]{BaHa16}. In conclusion, this algorithm is just a reinterpretation of the HKPV algorithm when $K(x,y)$ is hermitian, but which extends its applicability beyond the hermitian setting.

\begin{OP} As it is explained in \citep{HoKrPeVi06}, there are honest DPPs for which $K(x,y)$ is not a projection kernel, but a finite rank contraction, and for which it is still possible to  sample a point configuration. This however requires the explicit knowledge of the kernel's spectral decomposition, namely to have at disposal $\lambda_k$'s and two biorthogonal families $\phi_k,\psi_k$ such that, in $L^2(\mu)$,
$$
K(x,y)=\sum_{k=1}^{N}\lambda_k\phi_k(x)\psi_k(y).
$$
The idea is that if one samples for each $k\in\{1,\ldots, N\}$ a  Bernoulli random variables $X_k$ of parameter $\lambda_k\in(0,1]$  (recall  $K$ is a contraction), namely $\P(X_k=1)=1-\P(X_k=0)=\lambda_k$, and then sample a DPP from the finite rank projection kernel,
$$
K_I(x,y):=\sum_{k\in I}\phi_k(x)\psi_k(y),\qquad I:=\big\{k\in\{1,\ldots,N\} :\; X_k=1\big\},
$$
then one can check that the points obtained this way have the same distribution as the DPP with the original contraction kernel $K$. Could we find a procedure to sample  a contraction DPP without the a priori knowledge of its spectral decomposition? Note that one can always see a contraction as the restriction of a projection provided one enlarges the ambiant space; for instance, for finite state spaces, this amounts to couples the particles of the DPP with the holes where the particles are not. Can this be used to build an algorithm solely based on the data of the kernel $K(x,y)$?  

This question is  related to a natural but, to my knowledge, unsolved problem: Can we sample \emph{exactly} from the celebrated Sine kernel?  This is a projection kernel which is not of finite rank, hence generating a.s. an infinite number of points on $\R$. But if you fix a compact window, say $W:=[-10,10]$, and condition on the fact that the process generates, say, at most  $100$ points in that window (which is quite likely since this is a point process of intensity one), then one can ask about sampling such a restricted process in $W$. This leads to a contraction operator on $L^2(W)$ for which the spectral decomposition does not seem to be accessible. Same question for the Airy, Bessel, or Pearcey kernel. For instance, to my knowledge, there is no \emph{exact} algorithm to sample from the Tracy--Widom distribution yet.

\end{OP}

\bibliographystyle{plainnat} 
\bibliography{Recuref}

\begin{thebibliography}{22}
\providecommand{\natexlab}[1]{#1}
\providecommand{\url}[1]{\texttt{#1}}
\expandafter\ifx\csname urlstyle\endcsname\relax
  \providecommand{\doi}[1]{doi: #1}\else
  \providecommand{\doi}{doi: \begingroup \urlstyle{rm}\Url}\fi

\bibitem[Bardenet and Hardy(2016)]{BaHa16}
R.~Bardenet and A.~Hardy.
\newblock Monte carlo with determinantal point processes.
\newblock \emph{Preprint arXiv:1605.00361}, page 48p, 2016.

\bibitem[Berman(2016)]{Ber16}
R.~J. Berman.
\newblock Determinantal point processes and fermions on complex manifolds: Bulk
  universality.
\newblock \emph{To appear in Algebraic and Analytic Microlocal Analysis.
  Preprint ArXiv:0811.3341v2}, 2016.

\bibitem[Borodin(1999)]{Bor99}
A.~Borodin.
\newblock Biorthogonal ensembles.
\newblock \emph{Nuclear Phys. B}, 536\penalty0 (3):\penalty0 704--732, 1999.

\bibitem[Breuer and Duits(2017)]{BrDu13}
J.~Breuer and M.~Duits.
\newblock Central limit theorems for biorthogonal ensembles and asymptotics of
  recurrence coefficients.
\newblock \emph{J. Amer. Math. Soc.}, 30\penalty0 (1):\penalty0 27--66., 2017.

\bibitem[Haagerup and Thorbj{\o}rnsen(2012)]{HaTh12}
U.~Haagerup and S.~Thorbj{\o}rnsen.
\newblock Asymptotic expansions for the gaussian unitary ensemble.
\newblock \emph{Infin. Dimens. Anal. Quantum Probab. Relat. Top.}, 15\penalty0
  (1):\penalty0 1250003, 41 pp., 2012.

\bibitem[Hardy(2015)]{Har15}
A.~Hardy.
\newblock Average characteristic polynomials of determinantal point processes.
\newblock \emph{Ann. Inst. H. Poincare Probab. Statist.}, 51\penalty0
  (1):\penalty0 283--303, 2015.

\bibitem[Hough et~al.(2006)Hough, Krishnapur, Peres, and Vir{\'a}g]{HoKrPeVi06}
J.~B. Hough, M.~Krishnapur, Y.~Peres, and B.~Vir{\'a}g.
\newblock Determinantal processes and independence.
\newblock \emph{Probability Surveys}, 3:\penalty0 206---229, 2006.

\bibitem[Johansson(2006)]{Joh06}
K.~Johansson.
\newblock \emph{Random matrices and determinantal processes, Mathematical
  Statistical Physics}.
\newblock Elsevier B.V. Amsterdam, 2006.

\bibitem[K\"{o}ning(2005)]{Kon05}
W.~K\"{o}ning.
\newblock Orthogonal polynomial ensembles in probability theory.
\newblock \emph{Probab. Surv.}, 2\penalty0 (385--447), 2005.

\bibitem[Kuijlaars(2010)]{Kui10}
A.~B.~J. Kuijlaars.
\newblock \emph{Multiple orthogonal polynomial ensembles. In Recent Trends in
  Orthogonal Polynomials and Approximation Theory. Contemp. Math.}
\newblock Number 507. Amer. Math. Soc., Providence, RI, 2010.

\bibitem[Kuijlaars(2016)]{Kui16}
A.~B.~J. Kuijlaars.
\newblock \emph{Transformations of polynomials ensembles, in "Modern Trends in
  Constructive Function Theory" (D.P. Hardin, D.S. Lubinsky and B. Simanek,
  eds.)}.
\newblock Contemp. Math. 661, 2016.

\bibitem[Kuijlaars and Van~Assche(1999)]{KuVA99}
A.~B.~J. Kuijlaars and W.~Van~Assche.
\newblock The asymptotic zero distribution of orthogonal polynomials with
  varying recurrence coefficients.
\newblock \emph{J. Approx. Theory}, 99:\penalty0 167--197., 1999.

\bibitem[Lambert(2015)]{Lam15b}
G.~Lambert.
\newblock \textsc{CLT} for biorthogonal ensembles and related combinatorial
  identities.
\newblock \emph{Preprint arXiv:1511.06121}, 2015.

\bibitem[Lavancier et~al.(2015)Lavancier, M$\o$ller, and Rubak]{LaMoRu15}
F.~Lavancier, J.~M$\o$ller, and E.~Rubak.
\newblock Determinantal point process models and statistical inference.
\newblock \emph{Journal of Royal Statistical Society: Series B (Statistical
  Methodology)}, 77:\penalty0 853--877, 2015.

\bibitem[Ledoux(2004)]{Led04}
M.~Ledoux.
\newblock Differential operators and spectral distributions of invariant
  ensembles from the classical orthogonal polynomials. the continuous case.
\newblock \emph{Electronic Journal of Probability}, 9\penalty0 (7):\penalty0
  177--208, 2004.

\bibitem[Ledoux(2005)]{Led05}
M.~Ledoux.
\newblock Differential operators and spectral distributions of invariant
  ensembles from the classical orthogonal polynomials: The discrete case.
\newblock \emph{Electronic Journal of Probability}, 10\penalty0 (34):\penalty0
  1116--1146, 2005.

\bibitem[Lyons(2003)]{Lyo03}
R.~Lyons.
\newblock Determinantal probability measures.
\newblock \emph{Publ. Math. Inst. Hautes Etudes Sci.}, 98:\penalty0 167--212,
  2003.

\bibitem[Pastur and Shcherbina(2011)]{PaSh11}
L.~Pastur and M.~Shcherbina.
\newblock \emph{Eigenvalue distribution of large random matrices}, volume 171
  of \emph{Mathematical Surveys and Monographs}.
\newblock American Mathematical Society, Providence, RI, 2011.

\bibitem[Scardicchio et~al.(2009)Scardicchio, Zachary, and Torquato]{ScZaTo09}
A.~Scardicchio, C.~E. Zachary, and S.~Torquato.
\newblock Statistical properties of determinantal point processes in
  high-dimensional euclidean spaces.
\newblock \emph{Phys. Rev. E}, (3) 79\penalty0 (4, 041108):\penalty0 19 pp.,
  2009.

\bibitem[Simon(2009)]{Sim09}
B.~Simon.
\newblock Weak convergence of cd kernels and applications.
\newblock \emph{Duke Math. J.}, 146:\penalty0 305--330, 2009.

\bibitem[Simon(2011)]{Sim11}
B.~Simon.
\newblock \emph{Szeg\H{o}'s Theorem and its Descendants: Spectral Theory for
  $L^2$ Perturbations of Orthogonal Polynomials}.
\newblock M. B. Porter Lecture Series, Princeton Univ. Press, Princeton, NJ,
  2011.

\bibitem[Soshnikov(2000)]{Sos00b}
A.~Soshnikov.
\newblock Determinantal random point fields.
\newblock \emph{Russian Math. Surveys}, 55:\penalty0 923--975, 2000.

\end{thebibliography}

\end{document}